\documentclass[a4paper,12pt]{article}
\usepackage{amsmath,amsfonts,amssymb,amsthm, amsrefs} 

\newcommand{\Rset}{\mathbb{R}}

\newcommand{\PP}{\mathbb{P}}
\newcommand{\EE}{\mathbb{E}}
\newcommand{\DS}{\displaystyle} 

\newtheorem{Theorem}{Theorem}[section]

\newtheorem{Proposition}[Theorem]{Proposition}
\newtheorem{Lemma}[Theorem]{Lemma}

\newtheorem{Hypothesis}[Theorem]{Hypothesis}

\title{Differentiable perturbations of Ornstein-Uhlenbeck operators}
\author{Luigi Manca\\Dipartimento di Matematica P. e A.,\\Universit\`a di Padova,\\Via Trieste 63,\\ 35121 Padova, Italy\\ \emph{E-mail: manca@math.unipd.it}}
\begin{document}

\maketitle

\begin{abstract}
We prove an extension theorem for a small perturbation  of the Ornstein-Uhlenbeck operator $(L,D(L))$ in the space of all uniformly continuous and bounded functions $f:H\to \Rset$, where $H$ is a separable Hilbert space.
We consider a perturbation of the form $N_0\varphi=L\varphi+\langle D\varphi,F\rangle$ where $F:H\to H$ is bounded and Fr\'echet differentiable with uniformly continuous and bounded differential.
Hence, we prove that $N_0$ is $m$-dissipative and its closure in $C_b(H)$ coincides with the infinitesimal generator of a diffusion semigroup  associated to a stochastic differential equation in $H$.
\end{abstract}
\section{Introduction and setting of the problem}
Let $H$ be a separable Hilbert space endowed with scalar product $\langle \cdot,\cdot \rangle $ and norm $|\cdot|$.
We shall always identify $H$ with its topological dual space $H^*$.
$\mathcal L(H)$ is the Banach space of all the linear and continuous maps in $H$, endowed with the usual norm $\|\cdot\|_{\mathcal L(H)}$.
With $C_b(H)$ (resp. $C_b(H;H)$) we denote the Banach space of all   uniformly continuous and bounded functions $f:H\to \Rset$ (resp. $f:H\to H)$, endowed with the supremum norm $\|\cdot \|$ (resp. $\|\cdot \|_0)$.
We also denote by $C_b^1(H)$ (resp. $C_b^1(H;H)$) the space of all $f\in C_b(H)$ (resp. $f:H\to H$) that are Fr\'echet differentiable with differential in $C_b(H;H)$ (resp. with uniformly continuous and bounded differential $Df:H\to \mathcal L(H)$).
We assume the following
\begin{Hypothesis}
\label{h.3.1}
\begin{itemize}
\item[(i)] $A\colon D(A)\subset H\to H$ is the infinitesimal
generator of a strongly continuous semigroup $(e^{tA})_{t\geq0}$ of type ${\cal G}(1,\omega)$, i.e. there exists $\omega\in \Rset$ such that
\begin{equation}    \label{e.3.2}
   \|e^{tA}\|_{{\mathcal L}(H)}\le e^{\omega t},\quad t\geq 0;
\end{equation}
\item[(ii)] $Q\in \mathcal L(H)$ is self adjoint and positive;  
\item[(iii)]  For any $t> 0$  the linear operator $Q_t,$ defined as  
 \begin{equation}   \label{e.3.3}
   Q_tx=\int_0^te^{sA}Qe^{sA^*}xds,\;\;x\in H,\;t\ge 0,
 \end{equation}
  is of trace class.
\item[(iv)] $F\in C_b^1(H;H)$, and $\DS K=\sup_{x\in H}\|DF(x)\|_{\mathcal L(H)}$.
\end{itemize}
\end{Hypothesis}
It is well known (see, for instance, \cite{DPZ92}) that thanks to conditions (i)--(iii) it is possible to define the so called {\em Ornstein-Uhlenbeck} (OU) semigroup $(R_t)_{t\geq0}$ in $C_b(H)$ by the formula
\begin{equation} \label{e.OUS}
   R_t\varphi(x) = \int_0\varphi(e^{tA}x+y)N_{Q_t}(dy),\quad x\in H,
\end{equation}
where $N_{Q_t}$ is the Gaussian measure on $H$ of mean $0$ and covariance operator $Q_t$ (see \cite{DPZ92}).
It turns out that the OU semigroup is not a strongly continuous semigroup in $C_b(H)$ but it is a weakly continuous semigroup (see \cite{Cerrai}) and a $\pi$-semigroup (see \cite{Priola}).
However, it is possible to define its infinitesimal generator in the weaker sense
\begin{equation} \label{e.2.1}
\begin{cases}
  \DS D(L)=\bigg\{ \varphi \in C_b(H): \exists g\in C_b(H), \lim_{t\to 0^+} \frac{R_t\varphi(x)-\varphi(x)}{t}=g(x),  \\ 
  \DS  \qquad\qquad\quad x\in H,\;\sup_{t\in(0,1)}\left\|\frac{R_t\varphi-\varphi}{t}\right\|<\infty \bigg\}\\   
   {}   \\
  \DS L\varphi(x)=\lim_{t\to 0^+} \frac{R_t\varphi(x)-\varphi(x)}{t},\quad \varphi\in D(L),\,x\in H.
\end{cases}
\end{equation}
We are interested in the operator $(N_0,D(N_0))$ defined by
$$
    N_0\varphi=L\varphi + \mathcal F\varphi,\quad \varphi\in D(N_0)=D(L)\cap C_b^1(H),
$$
where
$$
    \mathcal F\varphi(x)=\langle D\varphi(x),F(x)\rangle.
$$ 
Now let us consider the stochastic differential equation in $H$
\begin{equation} \label{e.5}
\begin{cases}
   dX(t)= \big(AX(t)+F(X(t))\big)dt+Q^{1/2}dW(t) & t>0,\\
   X(0)=x & x\in H,
\end{cases}
\end{equation}
where  $(W(t)_{t\geq0}$ is a cylindrical Wiener process defined on a stochastic basis $(\Omega,\mathcal F,  (\mathcal F_t)_{t\geq0},\PP)$.
Since $F\in C_b^1(H;H)$,  problem \eqref{e.5} has a unique mild solution $\bigl(X(t,x)\bigr)_{t\geq0,x\in H}$ (see \cite{DPZ92}), that is for any  $x\in H$ the process $(X(\cdot,x))_{t\geq0}$ is adapted to the filtration $(\mathcal F_t)_{t\geq0}$ and it is continuous in mean square, i.e.
\[
  \lim_{t\to s}\EE\bigl[ |X(t,x)-X(s,x)|^2\bigr]=0,\quad \forall s\geq0.
\]
This allows us to define a transition semigroup $(P_t)_{t\geq0}$ in $C_b(H)$, by setting 
$$
   P_t\varphi(x)=\EE\bigl[\varphi(X(t,x))\bigr],\quad t\geq0,\,\varphi\in C_b(H),\,x\in H. 
$$
The semigroup    $(P_t)_{t\geq0}$ is not strongly continuous in $C_b(H)$. 
However, it is a $\pi$-semigroup, and we can define its infinitesimal generator $(N,D(N))$ in the same way as for the OU semigroup
\begin{equation*} 
\begin{cases}
  \DS D(N)=\bigg\{ \varphi \in C_b(H): \exists g\in C_b(H), \lim_{t\to 0^+} \frac{P_t\varphi(x)-\varphi(x)}{t}=g(x),  \\ 
  \DS  \qquad\qquad\quad x\in H,\;\sup_{t\in(0,1)}\left\|\frac{P_t\varphi-\varphi}{t}\right\|<\infty \bigg\}\\   
   {}   \\
  \DS N\varphi(x)=\lim_{t\to 0^+} \frac{P_t\varphi(x)-\varphi(x)}{t},\quad \varphi\in D(N),\,x\in H.
\end{cases}
\end{equation*}

The main result of this paper is the following
\begin{Theorem} \label{t.1.2}
Let us assume that Hypothesis \ref{h.3.1} holds. 
Then, the operator $(N_0,D(N_0)$, defined by $D(N_0)=D(L)\cap C_b^1(H)$ and $N_0\varphi=L\varphi+\mathcal F\varphi$, $\forall \varphi\in D(N_0)$, is $m$-dissipative in $C_b(H)$ and its closure is the operator $(N,D(N))$.
\end{Theorem}
In \cite{DP03}, it is proved that Theorem \ref{t.1.2} holds with $F\in C_b^{1,1}(H;H)$, that is $F$ is Fr\'echet differentiable and its differential $DF:H\to \mathcal L(H)$ is Lipschitz continuous.

Perturbations of OU operators as been the object of several papers (see, for instance, \cites{DP03, DP04, Manca06, Zambotti00}).
Frequently, additional assumptions are taken on the OU operator in order to have $D(L)\subset C_b^1(H)$.
%

In order to prove Theorem \ref{t.1.2} we develope a technique introduced in \cite{DP03}.
The idea  is the following:
since $F\in C_b^1(H;H)$, there exists a unique solution  $\eta(\cdot,x)$ of the abstract Cauchy problem
$$
\begin{cases}
    \DS\frac{d}{d\varepsilon}\eta(\varepsilon,x)=F(\eta(\varepsilon,x)), &\varepsilon>0,\\
    \eta(0,x)=x, &x\in H.
\end{cases}
$$
Then, for any $\varepsilon>0$ we define the operators $\mathcal F_\varepsilon:C_b(H)\to C_b(H)$ and $N_\varepsilon:D(N_\varepsilon)\subset C_b(H)\to C_b(H)$ by setting   
\begin{eqnarray*}
    &&   {\cal F}_\varepsilon \varphi(x)= \frac1\varepsilon\big(\varphi(\eta(\varepsilon,x))-\varphi(x)\big),\\
&& \begin{cases}
          D(N_\varepsilon)=D(L)\cap C_b^1(H), &{}  \\
       N_\varepsilon \varphi= L\varphi+{\cal F}_\varepsilon\varphi, & \varphi \in  D(N_\varepsilon).
 \end{cases}
\end{eqnarray*}
By an approximation argument, we are able to prove that the operator $(N_0,D(N_0))$ is $m$-dissipative in $C_b(H)$. 
Then, by the Lumer-Phillips theorem, it will follow that  the closure of $(N_0,D(N_0))$ coincides with the operator $(N,D(N))$.

\subsection{Properties of ${\cal F}_\varepsilon$}
The following lemma collects some useful properties of $\eta$.
\begin{Lemma}
The following estimates hold
\begin{equation} \label{e3.60}
 |\eta(t,x)|\leq  
    e^{\|F\|_0t}|x|;
\end{equation}
\begin{equation} \label{e3.62}
|\eta(t,x)-\eta(t,y)|\leq e^{ K_{} t}|x-y|;
\end{equation}
\begin{equation} \label{e3.66}
|\eta(t,x)-x| 
      \leq c\|F\|_0 t
\end{equation}
\begin{equation} \label{e3.70}
\|\eta_x(t,x)\|_{\mathcal L(H)}\leq e^{ K_{}t}
\end{equation}
\begin{equation} \label{e3.72}
\|\eta_x(t,x)-\eta_x(t,y)\|_{\mathcal L(H)}\leq e^{ K_{} t }\theta_{DF}(e^{ K_{} t }|x-y|), 
\end{equation}
where $\theta_{DF}:\Rset^+\times \Rset^+\to \Rset^+$ is the modulus of continuity of $DF$. 
\end{Lemma}
\begin{proof}
\eqref{e3.60}, \eqref{e3.66}, \eqref{e3.70} have been proved in \cite{DP03}*{Lemma 2.1}.\\
\eqref{e3.62}. We have  
\begin{eqnarray*}
 &&|\eta(t,x)-\eta(t,y)|\\
&&\leq |x-y|+\int_0^t\big|F(\eta(s,x))-F(\eta(s,y))|ds \\
&&\leq  K_{}\int_0^t|\eta(s,x)-\eta(s,y)|ds.
\end{eqnarray*}
Then \eqref{e3.62}   
follows by Gronwall's Lemma. \\
\eqref{e3.72}.
Let $x,y,h\in H$ and set
$$
  r^h(t)= \eta_x(t,x)\cdot h - \eta_x(t,y)\cdot h=p^h(t,x)-p^h(t,y),
$$  
where $P^h(t,x)=\eta_x(t,x)\cdot h$ and $p^h(t,y)=\eta_x(t,y)\cdot h$.
Then $r^h(t)$ fulfills the following equation
$$
 \begin{cases}
           \DS    \frac{d}{dt}r^h(t)= DF(\eta(t,x))r^h(t)+\big[ DF(\eta(t,x))-DF(\eta(t,y))\big] p^h(t,x), & t>0 \\{}\\
            r^h(0)=0. & {}
 \end{cases}
$$
 Since $|DF(\eta(t,x))r^h(t)|\leq  K_{}|r^h(t)|$ it follows that $r^h(t)$ is bounded by
 $$
   |r^h(t)|\leq \int_0^t e^{ K_{}(t-s)}\big\|DF(\eta(s,x))-DF(\eta(s,y))\big\|_{\mathcal L(H)}|p^h(s,x)|ds.
 $$
 By taking into account that $DF:H\to \mathcal L(H;H)$ is uniformly continuous and bounded, we denote by $\theta_{DF}$ the modulus of continuity of $DF$.  
Hence, by \eqref{e3.62}, \eqref{e3.70} we have
\begin{eqnarray*} 
   |r^h(t)|&\leq& \int_0^t e^{ K_{} s }\theta_{DF}(|\eta(s,x)-\eta(s,y)|)ds|h|\\
   &\leq& e^{ K_{} t }\theta_{DF}( e^{ K_{} t }|x-y|)|h|
\end{eqnarray*}
 \end{proof}
\begin{Proposition}
  For any $\varphi \in C_b^1(H)$ we have
\begin{eqnarray} 
  && \lim_{\varepsilon \to 0^+} \mathcal F_\varepsilon \varphi = 
    \mathcal F \varphi\quad\text{in }C_b(H). \label{e3.13}\\
    && \|\mathcal F_\varepsilon \varphi\|_{} \leq \|D\varphi\|_0\|F\|_0. \label{e3.14}
  \end{eqnarray}
\end{Proposition}
\begin{proof}
For all $\varphi \in  C_b^1(H)$ we have
\begin{eqnarray*}
   \mathcal F_\varepsilon \varphi(x) - \mathcal F \varphi(x) 
   &=& \frac{1}{\varepsilon}\int_0^\varepsilon 
      \left\langle D\varphi(\eta(s,x))-D\varphi(x),F(\eta(s,x))\right\rangle ds \\
   && +  \frac{1}{\varepsilon}\int_0^\varepsilon 
      \left\langle D\varphi(x),F(\eta(s,x))-F(x)\right\rangle ds.
\end{eqnarray*}
Then by \eqref{e3.66} we have
\begin{eqnarray*}
   && |\mathcal F_\varepsilon \varphi(x) - \mathcal F \varphi(x)| \leq \\
   && \leq \frac{1}{\varepsilon}\int_0^\varepsilon 
     \left( |\theta_{D\varphi}(|\eta(s,x)-x|)\|F\|_0
     + \|D\varphi\|_0  K_{}|\eta(s,x)-x|) \right) ds \\
   && \leq \frac{1}{\varepsilon} \int_0^\varepsilon \left( \theta_{D\varphi}(\|F\|_0s|)\|F\|_0
     + \|D\varphi\|_0  K_{}\|F\|_0s \right) ds\\
   &&\leq \left( \theta_{D\varphi}(\|F\|_0\varepsilon|)
     + \|D\varphi\|_0  K_{}\varepsilon \right)\|F\|_0
\end{eqnarray*}
  where $\theta_{D\varphi}$, is the modulus of continuity of $D\varphi$.
This yields \eqref{e3.13}.
Moreover, we have
$$
   \mathcal F_\varepsilon \varphi(x) =  \frac{1}{\varepsilon}\int_0^\varepsilon 
      \left\langle D\varphi(\eta(s,x)),F(\eta(s,x))\right\rangle ds
$$
that implies \eqref{e3.14}.
\end{proof}
%
%
%
    \subsection{$m$-dissipativity of $N$.}
Given $\varepsilon >0$ we introduce the following approximating operator
$$
    N_\varepsilon = L+ \mathcal F_\varepsilon,\,D(N_\varepsilon)=D(L)\cap C_b^1(H).
$$
We have
\begin{Proposition}     \label{p3.7b}
$N_\varepsilon $ is a $m$-dissipative operator in $C_b(H)$ for any $\varepsilon >0$.
Moreover, for any $f\in  C_b^1(H)$ and any $\lambda > \omega+(e^{\varepsilon  K_{}}-1)/\varepsilon$
the operator
$$
   R(\lambda,N_\varepsilon) =\big( 1-T_\lambda)^{-1}R\left(\lambda+\frac1\varepsilon,L\right) ,
$$
where $T_\lambda : C_b(H)\to C_b(H)$ is defined by
\begin{equation}\label{e.14a}
   T_\lambda\psi(x)= R\left(\lambda+\frac{1}{\varepsilon},L\right)\left[\frac{1}{\varepsilon}\psi(\eta(\varepsilon,x))\right],\quad x\in H,\, \psi \in  C_b(H)
\end{equation}
maps $C_b^1(H)$ into $D(L)\cap C_b^1(H)$
%
and
\begin{equation} \label{e3.16}
    \|DR(\lambda, N_\varepsilon)f\|_0 \leq \frac{1}{\lambda -\omega -\frac{e^{ K_{}\varepsilon}-1}{\varepsilon}  }\|Df\|_0.
\end{equation}
\end{Proposition}
\begin{proof}
Let $\varepsilon>0$, $\lambda >0$, $f\in C_b(H)$.
The equation 
$$
    \lambda\varphi_\varepsilon - L \varphi_\varepsilon - \mathcal F(\varphi_\varepsilon) = f
$$
is equivalent to 
$$
   \left( \lambda+\frac{1}{\varepsilon}\right)\varphi_\varepsilon - L \varphi_\varepsilon - \mathcal F(\varphi_\varepsilon) = f+\frac{1}{\varepsilon}\varphi_\varepsilon(\eta(\varepsilon,\cdot))
$$
and to 
\begin{equation} \label{e3.18}
    \varphi_\varepsilon=R\left(\lambda+\frac{1}{\varepsilon},L\right)f+ T_\lambda\varphi_\varepsilon.
\end{equation}
Since, as we can easily see, for any $\lambda>0$ 
\begin{equation} \label{e.15a}
   \| T_\lambda\psi\|_{} \leq \frac{1}{1+\lambda\varepsilon}\|\psi\|_{},\quad \forall \psi \in C_b(H),
\end{equation}
 the operator $T_\lambda$ is a contraction in $C_b(H) $ and so equation \eqref{e3.18} has a unique solution $\varphi_\varepsilon \in C_b(H)$ done by $\varphi_\varepsilon=R(\lambda,N_\varepsilon)f$.
Moreover, by \eqref{e.14a}, \eqref{e.15a} it holds
$$
   \|\varphi_\varepsilon\|_{} \leq \frac{1}{\lambda+ \frac{1}{\varepsilon}}\left[ \|f\|_{}+ \frac{1}{\varepsilon}\|\varphi_\varepsilon\|_{}\right].
$$
Consequently,
$$
    \|\varphi_\varepsilon\|_{} \leq  \frac{1}{\lambda}\|f\|_{}.
$$
Then, $N_\varepsilon$ is $m$-dissipative.
Now let $f\in C_b^1(H)$.
  We recall that for any $\lambda>0$, $\psi\in C_b(H)$ 
\begin{equation}\label{e.16a}
  R(\lambda,L)\psi(x)=\int_0^\infty e^{-\lambda t}R_t\psi(x)dt
\end{equation}
  and that
$$
    DR_t\psi(x)=\int_He^{tA^*}D\psi(e^{tA}x+y)N_{Q_t}(dy).
$$
Hence, for any $\lambda >\omega$ 
\begin{equation} \label{e3.18b}
    DR(\lambda,L)\psi(x)=   \int_0^\infty \int_H e^{-\lambda t}e^{tA^*}D\psi(e^{tA}x+y)N_{Q_t}(dy)dt
\end{equation} 
and so
\begin{equation} \label{e3.19}
   \|DR(\lambda,L)\psi\|_0 \leq \frac{1}{\lambda-\omega} \|D\psi\|_0
\end{equation}
Moreover, as it can be easily seen by \eqref{e3.18b}, $DR(\lambda,L)\psi$ is uniformly continuous. 
Then $R(\lambda,L):C_b^1(H)\to C_b^1(H)$.
Now, in order to prove  that $T_\lambda:C_b^1(H)\to C_b^1(H)$
it is sufficient to show that $\psi(\eta(\varepsilon,x)) \in C_b^1(H)$, for any $\psi \in C_b^1(H)$.
Indeed, by a standard computation, we have
$$
  D\psi(\eta(\varepsilon,\cdot))(x) = \eta_x^*(\varepsilon,x)D\psi(\eta(\varepsilon,x)),\quad x\in H.
$$
Consequently, by \eqref{e3.62}, \eqref{e3.72} we have
\begin{eqnarray*}
  && |D\psi(\eta(\varepsilon,\cdot))(x) - D\psi(\eta(\varepsilon,\cdot))(\overline{x})| \\
  &\leq&   \| \eta_x^*(\varepsilon,x)- \eta_x^*(\varepsilon,\overline{x})\|_{\mathcal L(H)} |D\psi(\eta(\varepsilon,x))| \\
    &&  + \|\eta_x^*(\varepsilon,\overline{x})\|_{\mathcal L(H)} |D\psi(\eta(\varepsilon,x))-D\psi(\eta(\varepsilon,\overline{x}))|   \\
  &\leq&   e^{\varepsilon  K_{}}\theta_{DF}
(e^{\varepsilon  K_{}}|x-\overline{x}|)\|D\psi\|_0+ e^{\varepsilon K_{}}\theta_{D\psi}(|\eta(\varepsilon,x) -\eta(\varepsilon,\overline{x})|)   \\
   &\leq &  e^{\varepsilon  K_{}} \theta_{DF} (e^{\varepsilon  K_{}}|x-\overline{x}|)\|D\psi\|_0+e^{\varepsilon  K_{}}\theta_{D\psi}(e^{\varepsilon  K_{}}|x-\overline{x}|),
\end{eqnarray*}
for any $x, \overline{x}\in H$.
So, $DT_\lambda \psi(\cdot)$ is uniformly continuous.
Now we prove that $T_\lambda$ is a contraction in $C_b^1(H)$.
By \eqref{e.14a}, \eqref{e.16a} we have
\begin{eqnarray*}
  T_\lambda\psi(x)&=&\frac{1}{\varepsilon}\int_0^\infty e^{-(\lambda+\frac{1}{\varepsilon}) t}R_t\psi(\eta(\varepsilon,\cdot))(x)dt \\
  &=&  \frac{1}{\varepsilon} \int_0^\infty \int_H e^{-(\lambda+\frac{1}{\varepsilon}) t} \psi(\eta(\varepsilon,e^{tA}x+y))N_{Q_t}(dy)dt
\end{eqnarray*}
Then
$$
  DT_\lambda\psi(x)=  
$$
$$
  =  \frac{1}{\varepsilon}  \int_0^\infty \int_H e^{-(\lambda+\frac{1}{\varepsilon}) t}e^{tA^*} \eta_x^*(\varepsilon,e^{tA}x+y)D\psi(\eta(\varepsilon,e^{tA}x+y))N_{Q_t}(dy)dt
$$
By \eqref{e3.70} it follows
\begin{eqnarray*}
  |DT_\lambda\psi(x)| &\leq& \frac{1}{\varepsilon}  \int_0^\infty e^{-(\lambda+\frac{1}{\varepsilon}-\omega) t}e^{\varepsilon K_{}}\|D\psi\|_0 dt \\
  &=&   \frac{e^{\varepsilon K_{}}}{1+\varepsilon(\lambda-\omega)} \|D\psi\|_0.
\end{eqnarray*}
Therefore, for any $\lambda > \omega+(e^{\varepsilon K_{}}-1)/\varepsilon$ the linear operator $T_\lambda$ is a contraction in $C_b^1(H)$ and its resolvent satisfies
$$
 (1-T_\lambda)^{-1}(C_b^1(H))\subset C_b^1(H),
$$
\begin{equation} \label{e3.20}
   \|D(1-T_\lambda)^{-1}\psi\|_0\leq \frac{1}{  1-\frac{e^{\varepsilon K_{}}}{1+\varepsilon(\lambda-\omega)}}\|D\psi\|_0.
\end{equation}
This implies
$$
  R(\lambda,N_\varepsilon)(C_b^1(H))=(1-T_\lambda)^{-1}R\Big(\lambda+\frac1\varepsilon,L\Big)(C_b^1(H))\subset C_b^1(H).
$$
Then, $N_\varepsilon$ is $m$-dissipative.
Finally, \eqref{e3.16} follows by \eqref{e3.19} and \eqref{e3.20}.
\end{proof}
\begin{Lemma}
   The operator $N_0$ is dissipative in $C_b(H)$.
\end{Lemma}
\begin{proof}
 We have to prove that $\|\lambda \varphi-N_0\varphi\|_{}\geq \lambda\|\varphi\|_{}$ for any $\varphi \in D(N_0)$, $\lambda>0$.
 So, if $\varphi \in D(L)\cap C_b^1(H)$ and $\lambda >0$ we set
 $$
   \lambda \varphi -L\varphi -\mathcal F\varphi=f.
 $$
 then for any $\varepsilon  >0$  we have
 $$
    \lambda\varphi-N_\varepsilon\varphi = f+\mathcal F\varphi-\mathcal F_\varepsilon\varphi.
 $$
 It follows 
 $$
   \varphi = R(\lambda, N_\varepsilon)( f+\mathcal F\varphi-\mathcal F_\varepsilon\varphi)
 $$
 and
 $$
     \|\varphi\|_{} \leq \frac1\lambda (\|f\|_{}+\| \mathcal F\varphi-\mathcal F_\varepsilon\varphi\|_{})
 $$
 Then by \eqref{e3.13} it follows 
 $$
      \|\varphi\|_{} \leq \frac1\lambda \|f\|_{}.
 $$
\end{proof}
Since $N_0$ is dissipative, its closure $\overline{N}_0$  is still dissipative (maybe it is multivalued).
By the following theorem follows Theorem \ref{t.1.2}.
\begin{Theorem}  \label{t3.9}
$N_0$ is $m$-dissipative.
\end{Theorem}
\begin{proof}
Let $f\in C_b^1(H)$, $\varepsilon\in (0,1)$ and $\lambda>\omega + e^{ K_{}}-1$.
We denote by $\varphi_\varepsilon$ the solution of
$$
\lambda \varphi_\varepsilon-N_\varepsilon\varphi_\varepsilon=f.
$$
By Proposition \eqref{p3.7b} we have $ \varphi_\varepsilon \in D(L)\cap C_b^1(H)=D(N_0)$, 
  then $\varphi_\varepsilon$ is solution  of
$$
   \lambda \varphi_\varepsilon-N_0\varphi_\varepsilon=f+\mathcal F_\varepsilon \varphi_\varepsilon - \mathcal F \varphi_\varepsilon.
$$
We claim that $\mathcal F_\varepsilon \varphi_\varepsilon - \mathcal F \varphi_\varepsilon \to 0$ in $C_b(H)$ as $\varepsilon\to0^+$.
Indeed it holds
$$
   \mathcal F_\varepsilon \varphi_\varepsilon(x) - \mathcal F \varphi_\varepsilon(x) =
$$
$$
    \frac1\varepsilon\int_0^\varepsilon \left( \langle D\varphi_\varepsilon(\eta(s,x)),F(\eta(s,x))\rangle +\langle  D\varphi_\varepsilon(x),F(x)\rangle   \right)ds 
$$
$$
  =  \frac1\varepsilon\int_0^\varepsilon \left(\langle D\varphi_\varepsilon(\eta(s,x))-D\varphi_\varepsilon(x),F(\eta(s,x))\rangle\right.
$$
$$
   +  \left.\langle  D\varphi_\varepsilon(x), F(\eta(s,x))-F(x)\rangle   \right)ds .
$$
Hence
$$
   |\mathcal F_\varepsilon \varphi_\varepsilon(x) - \mathcal F \varphi_\varepsilon(x)| \leq
$$
$$
  \leq   \frac1\varepsilon\int_0^\varepsilon \left(|D\varphi_\varepsilon(\eta(s,x))-D\varphi_\varepsilon(x)|\|F\|_0+\|D\varphi_\varepsilon\|_0| F(\eta(s,x))-F(x)|\right)ds
$$
By \eqref{e3.66} we have
$$
    | F(\eta(s,x))-F(x)|\leq  K_{}| \eta(s,x)-x|
    \leq   K_{}\|F\|_0s\leq   K_{}\|F\|_0\varepsilon.
$$
Notice now that since $\varphi_\varepsilon=R(\lambda,N_\varepsilon)f$ and $\varepsilon \in (0,1)$, by \eqref{e3.16} it follows
$$
  \|D\varphi_\varepsilon\|_0\leq c_1\|Df\|_0,
$$
for all $\varepsilon \in (0,1)$, where $c_1=(\lambda-\omega- K_{}e^{ K_{}})^{-1}$.
This also implies
$$
     |D\varphi_\varepsilon(\eta(s,x))-D\varphi_\varepsilon(x)\|_0     \leq c_1\|Df(\eta(s,x)+\cdot)-Df(x+\cdot)\|_0 
$$
$$
   \leq  c_1 |\theta_{Df}(|\eta(s,x)-x|)  \leq c_1 \theta_{Df}(\|F\|_0\varepsilon),
$$
where $\theta_{Df}:\Rset^+\to \Rset^+$ is the modulus of continuity of $Df$.
So we find
$$
     |\mathcal F_\varepsilon \varphi_\varepsilon(x) - \mathcal F \varphi_\varepsilon(x)| \leq
$$
$$
        \leq   c_1 \|F\|_0 \theta_{Df}(\|F\|_0\varepsilon)+  c_1\|Df\|_0   K_{}\|F\|_0\varepsilon.
$$
Then $\mathcal F_\varepsilon \varphi_\varepsilon- \mathcal F \varphi_\varepsilon\to 0$ in $C_b(H)$, as $\varepsilon  \to  0^+$.
Finally, we have obtained
$$
   \lim_{ \varepsilon  \to  0^+} \big[\lambda\varphi_\varepsilon-N_0\varphi_\varepsilon]=f
$$
in $C_b(H)$. 
Therefore the closure of the range of $\lambda-N$ includes $C_b^1(H)$, which is dense in $C_b(H)$.
So, since $N_0$ is dissipative, by the Lumer-Phillips theorem the closure $\overline{N}_0$ of $N_0$ is $m$-dissipative.
\end{proof}
\subsection{Proof of Theorem \ref{t.1.2}}
By Theorem \ref{t3.9} the operator $N_0$ is $m$-dissipative in $C_b(H)$.
It is also known that if $\varphi \in D(L)\cap C_b^1(H)$, then $N\varphi=L\varphi+\mathcal F\varphi$ (see, for instance, \cite{Manca07}) and therefore $(N,D(N))$ is an extension of $(N_0,D(N_0))$.
Finally, since the operator $(N,D(N))$ is closed (see Proposition 3.4 in \cite{Priola}), by the Lumer-Phillips theorem it follows that the closure of $(N_0,D(N_0))$ in $C_b(H)$ coincides with $(N,D(N))$.  \qed



\end{document}